\documentclass{article}

\usepackage[a4paper,outer=3cm]{geometry}
\usepackage{fullpage}
\DeclareOldFontCommand{\bf}{\normalfont\bfseries}{\mathbf}

\usepackage[utf8]{inputenc}
\usepackage[T1]{fontenc}
\usepackage{times}

\usepackage{xspace}

\usepackage{longtable}
\usepackage{amsmath,amsthm,amssymb,amsthm, stmaryrd}
\usepackage{enumerate}
\usepackage{multirow}
\usepackage{algorithm}
\usepackage{algpseudocode}
\algnewcommand\algorithmicinput{\textbf{Input:}}
\algnewcommand\algorithmicoutput{\textbf{Output:}}
\algnewcommand\Input{\item[\algorithmicinput]}
\algnewcommand\Output{\item[\algorithmicoutput]}

\usepackage{hyperref}

\usepackage{pgf,tikz}
\usepackage{svg}
\usetikzlibrary{arrows}
\usetikzlibrary{decorations.markings}

\usepackage{pdftricks}
\begin{psinputs}
	\usepackage{pstricks,pst-plot,pst-node,pst-func}
\end{psinputs}
\usepackage{graphics,graphicx}

\tikzstyle{vertex}=[circle, draw, inner sep=0pt, minimum size=6pt]
\newcommand{\vertex}{\node[vertex]}

\usetikzlibrary{decorations.pathreplacing}

\tikzset{->-/.style={decoration={
  markings,
  mark=at position .5 with {\arrow{>}}},postaction={decorate}}}

 \usepackage{marginnote}

\newcommand{\m}[1]{}

\usepackage[nothm]{thmbox}
\usepackage{amsthm}
\usepackage{thmtools}
\usetikzlibrary{patterns}
\usetikzlibrary{calc}

\declaretheorem[parent=section,thmbox=M]{theorem}

\declaretheorem[numberlike=theorem,thmbox=M]{corollary}
\declaretheorem[numberlike=theorem,thmbox=M]{conjecture}

\declaretheorem[numberlike=theorem,thmbox=M]{definition}

\newtheorem{claim}{Claim}[theorem]

\declaretheorem[numberlike=theorem]{lemma}

\newcommand{\Gya}{Gy\'arf\'as\xspace}

\newcommand{\Erd}{Erd\H os\xspace}

\newcommand{\mc}{\mathcal}

\newcommand{\Ra}{\Rightarrow}

\newcommand{\ora}[1]{\overrightarrow{#1}}

\DeclareMathOperator{\dic}{\ora \chi}
\DeclareMathOperator{\diomega}{\ora \omega}

\newcommand{\intv}[2]{\llbracket #1 , #2 \rrbracket}

\tikzstyle{vertex}=[circle,draw, top color=gray!5, 
	    bottom color=gray!30, minimum size=12pt, scale=1, inner sep=0.5pt]
\tikzstyle{arc}=[->, > = latex',  thick]
\tikzstyle{edge}=[thick, blue]

\def\centerarc[#1](#2)(#3:#4:#5) {\draw[#1] ($(#2)+({#5*cos(#3)},{#5*sin(#3)})$) arc (#3:#4:#5); }

\newenvironment{proofclaim}
	{\noindent {\bf Proof of Claim
	:}}
	{\hfill $\square$ \par\vspace{11pt}}

\usepackage{comment}
\usepackage{pythonhighlight}

\usepackage{authblk}

\title{Computing the clique number of tournaments}
\author{Guillaume Aubian\footnote{Supported by project 22-17398S (Flows and cycles in graphs on surfaces) of Czech Science
Foundation.}}
\date{}

\affil[]{Charles University, Prague, Czech Republic.}

\begin{document}

\maketitle

\begin{abstract}
The clique number of a tournament is the maximum clique number of a graph formed by keeping backwards arcs in an ordering of its vertices. 
We study the time complexity of computing the clique number of a tournament and prove that, for any integer $k \geq 3$, deciding whether a tournament has clique number at most $k$ is NP-complete. This answers an interrogation of Nguyen, Scott and Seymour. To do so, we make use of a construction which we then modify to provide a counterexample to a conjecture of Aboulker, Aubian, Charbit and Lopes.
\end{abstract}

\tableofcontents


\section{Introduction}


Given an undirected graph $G$, its clique number, denoted $\omega(G)$, is the maximum number of pairwise adjacent vertices of $G$. Its chromatic number, denoted $\chi(G)$, is the minimum size of a partition of $V(G)$ in stable sets. A tournament is an orientation of a complete graph.

The clique number is a central notion in graph theory, finding applications in domains as diverse as chemistry \cite{RWCDH03}, automatic test pattern generation \cite{HP98}, bioinformatics \cite{DS86} etc. Computing the clique number of a graph has been proven to be NP-complete \cite{K72}. However, for a fixed integer $k$, whether a graph on $n$ vertices has clique number at most $k$ can be decided in time complexity $O(n^k)$, by checking all subsets of at most $k$ vertices.

Given a tournament $T$ and an ordering $\prec$ of $V(T)$, we define the backedge graph of $T$ relative to $\prec$ as $T^\prec = (V(T), \{uv \in A(T) \mid v \prec u\})$.
In \cite{AACL23}, Aboulker, Aubian, Charbit and Lopes formally study the clique number of tournaments defining it as the minimum over all orderings $\prec$ of $V(T)$ of the clique number of $T^\prec$, and denoting it as $\diomega(T)$. Note that while they formally introduced it, the same notion was already studied in \cite{NSS23} by Nguyen, Scott and Seymour. In the beginning of Section~9, they wonder what is the time complexity of computing the clique number of tournaments.

In this paper, we prove that, akin to the undirected case, this problem is NP-complete:

\begin{theorem}
Computing the clique number of tournaments is NP-complete.
\end{theorem}

We prove a stronger theorem. Let $k \geq 1$, the problem \textsc{$k$-DIOMEGA-ORDERING} is defined as follows:

\bigskip

\textsc{$k$-DIOMEGA-ORDERING}

\textbf{Input:} A tournament $T$

\textbf{Output:} Does $\diomega(T) \leq k$ ?

\bigskip

Unlike the undirected case, we prove the following:
\begin{theorem}
Let $k \geq 3$. \textsc{$k$-DIOMEGA-ORDERING} is NP-complete.
\end{theorem}

Note that, since \textsc{$1$-DIOMEGA-ORDERING} is equivalent to testing whether a given tournament is acyclic, and thus is polynomial, only the case $k = 2$ remains open:

\begin{conjecture}
\textsc{$2$-DIOMEGA-ORDERING} is NP-complete.
\end{conjecture}

\bigskip

Every graph $G$ satisfies $\chi(G) \geq \omega(G)$. However, there exist graphs with $\chi(G) > \omega(G)$ -- for example a cycle on $5$ vertices. Another question is whether there exists a function $f$ such that for any graph $G$, $\chi \leq f(\omega(G))$. This question was answered in the negative by Zykov in \cite{Z49}. However, various graph classes $\mc C$ actually satisfy that for any graph $G \in \mc C$, $\chi(G) \leq f(\omega(G))$ for some function $f$. Such graph classes are said to be $\chi$-bounded.

\medskip

$\chi$-bounded classes of graphs have been widely studied, see for example \cite{SS20} for a nice survey on $\chi$-boundedness. Since most graph classes $\mc C$ that appear in the literature are preserved under taking induced subgraphs, and thus can be defined from a set $S$ as the class of graphs not containing a graph in $S$ as an induced subgraph, the following question arises naturally: at which condition on $S$ is $\mc C$ $\chi$-bounded? In \cite{E59}, \Erd proved that, if $S$ is finite and $\mc C$ is $\chi$-bounded, $S$ must contain at least one forest. \Gya and Sumner thus asked whether the converse holds:

\begin{conjecture}[\Gya, 1975 \& Sumner, 1981]
 For any forest $F$, the class of graphs not containing $F$ as an induced subgraph is $\chi$-bounded.
\end{conjecture}

This conjecture remains largely open. See \cite{SS20}.

\bigskip

In $1982$, Victor Neumann-Lara introduced in \cite{NL82} a directed analogue of the chromatic number: the dichromatic number of a directed graph $D$, denoted $\dic(D)$, is the minimum size of a partition of its vertices into acyclic subsets\cite{NL82}. This notion has been reintroduced by Mohar in \cite{M03}, and since then this notion has known a growing interest, mainly in trying to translate undirected results to a directed setting.

\medskip

Similarly to the undirected setting, Aboulker, Aubian, Charbit and Lopes have then defined in \cite{AACL23} a notion of $\chi$-boundedness for tournaments: a tournament class $\mc T$ is $\dic$-bounded if there exists a function $f$ such that for any tournament $T \in \mc T$, $\dic(T) \leq f(\diomega(T))$. A question that arises naturally is the following: which tournaments $T$ are such that the class of tournaments not containing $T$ are $\dic$-bounded? In a similar flavour to the directed case, they proved that such a $T$ must contain a backedge graph that is a forest. This directly implies an analogue of \Gya-Sumner conjecture for tournaments:

\begin{conjecture}
Let $T$ be a tournament with a backedge graph that is a forest. The class of tournaments not containing $T$ as a subgraph is $\dic$-bounded.
\end{conjecture}

We answer this question in the negative.


\section{Definitions and Notations}


Definitions and notations of this paper that are not explained in this section follow from classical textbooks such as \cite{BG18}, \cite{BM08} or \cite{D05}.

Given a set $X$ and an integer $n$, $X \choose n$ is the set of all subsets of $X$ of size exactly $n$.

Given two disjoint sets of vertices $X, Y$ of a digraph $D$, we write $X \Rightarrow Y$ to say that for every $x \in X$ and every $y \in Y$, $xy \in A(D)$, and we write $X \rightarrow Y$ to say that every arc with one end in $X$ and the other one in $Y$ is oriented from $X$ to $Y$ (but some vertices of $X$ might be non-adjacent to some vertices of $Y$). When $X=\{x\}$ we write $x \Rightarrow Y$ and $x  \rightarrow Y$. 

We also use the symbol $\Ra$ to denote a composition operation on digraphs: for two digraphs $D_1$ and $D_2$, $D_1\Ra D_2$ is the digraph obtained from the disjoint union of $D_1$ and $D_2$ by adding all arcs from $V(D_1)$ to $V(D_2)$.

\medskip 

A {\em transitive tournament} is an acyclic tournament and we denote by $TT_n$ the unique acyclic tournament on $n$ vertices.  

Given three tournaments $T_1,T_2,T_3$, we denote by $\Delta(T_1,T_2,T_3)$ the tournament obtained from  disjoint copies of $T_1,T_2,T_3$ by adding arcs in such a way that $T_1\Ra T_2$, $T_2\Ra T_3$ and $T_3\Ra T_1$. Suppose one or more of the tournaments $T_i$ is a transitive tournament $TT_k$. In that case, we simplify the notation by using its size $k$ instead of writing $TT_k$ in the $\Delta$ construction: for example, $\Delta(1,k, T)$ corresponds to $\Delta(TT_1, TT_k,T)$ and $\Delta(1,1,1)$ is simply the directed triangle, which we also denote by $C_3$.

\medskip

Given a graph or tournament with a total ordering $\prec$ of its vertex set $V$ and two disjoint subsets $A, B$ of $V$, we write $A \prec B$ to say that for every $a \in A$ and every $b \in B$, $a\prec b$. For a digraph or tournament with a total ordering $\prec$ of its vertex set $V$, an arc $uv$ such that $u\prec v$ is called {\em forward}; otherwise, it is called {\em backwards}. Recall that given a tournament $T$, and a total order $\prec$ on $V(T)$, the backedge graph $T^{\prec}$ of $T$ with respect to $\prec$ is the (undirected) graph with vertex set $V(T)$ and edges $uv$ if $u\prec v$ and $vu\in A(T)$ (i.e. $vu$ is backward).
An ordering $\prec$ such that $\omega(T^{\prec})= \diomega(T)$ is called an \emph{$\diomega$-ordering} of $T$.
For $\prec$ an ordering of a set $X$, we denote by $\prec_{\mid Y}$ the restriction of $\prec$ to a subset $Y \subseteq X$. For $T, T'$ two tournaments with $V(T) \subseteq V(T')$ and $\prec$ an ordering of $T'$, we sometimes use $T^\prec$ to denote $T^{\prec_{\mid V(T)}}$.

\section{Computing the clique number of tournaments is NP-complete}

In this section, we will be interested in the algorithmic aspects of computing the clique number of tournaments. Our main result is the following:

\begin{theorem}\label{thm:k_diomega_NP}
    Let $k \geq 3$. \textsc{$k$-DIOMEGA-ORDERING} is NP-complete.
\end{theorem}

which directly implies:

\begin{theorem}\label{thm:diomega_NP}
Computing the clique number of tournaments is NP-complete.
\end{theorem}

Note that while Theorem~\ref{thm:diomega_NP} is expected, as an analogue of the undirected case, Theorem~\ref{thm:k_diomega_NP} is surprising, in that it contrasts with the undirected case where the corresponding decision problem can be answered in polynomial time. To prove this, we will need intermediate results: this is the purpose of Subsection~\ref{subsec:tools}.

\subsection{Useful tools to increase the clique number}\label{subsec:tools}

Note that a variation of the construction in the following proof will be used in Definition~\ref{def:pi}.

\begin{lemma}\label{lem:cut}
Let $T$ be a tournament. There exists a tournament $T'$ with $\diomega(T') = \diomega(T)$ and such that for any $X \subseteq V(T')$, $T'[X]$ contains a copy of $T$ or $T'[\overline{X}]$ contains a copy of $T$.
\end{lemma}

\begin{proof}
    If $\diomega(T) = 1$, then $T$ is a transitive tournament, and thus $T' = T \Ra T$ has the desired property. Thus, we can suppose $\diomega(T) \geq 2$.

    Let $v_1, \dots, v_n$ be a $\diomega$-ordering of $V(T)$.
Let $m = {{n(n-1)+1} \choose n}$ and let $T_1, \dots, T_{m \times n}$ be $m \times n$ copies of $T$.
    Let $i \in \intv{1}{n}$. Let $B_i = \intv{1 + (i-1) \times m}{i \times m}$.
    Note that $(B_i)_{i \in \intv{1}{n}}$ is a partition of $\intv{1}{m \times n}$, thus we can associate to each $x \in \intv{1}{m \times n}$ an integer $B^{-1}(x)$ such that $x \in B_{B^{-1}(x)}$.
    Let $\varphi_i$ be a bijection from $B_i$ to ${\intv{1}{n\times(n-1) + 1}} \choose n$ and for $j \in B_i$, let $\psi_j$ be a bijection from $V(T_j)$ to $\varphi_i(j)$.
    Thus, we associate to each $j \in B_i$ a list of $n$ integers $\varphi_i(j)$, and to each $v \in V(T_j)$ an element in this list.
    $T'$ will be obtained from $T_1 \Ra \dots \Ra T_m$ by reversing arc $uw$ where $u \in T_i$ and $w \in T_j$ if and only if:
    \begin{itemize}
        \item $i < j$,
        \item $\psi_i(u) = \psi_j(w)$,
        \item $v_{B^{-1}(j)} v_{B^{-1}(i)} \in A(T)$.
    \end{itemize}

    Let us first prove that $\diomega(T') = \diomega(T)$. Let $\prec$ be the ordering of $V(T')$ such that if $u$ is the copy of vertex $v_i$ in $T_j$ and $w$ is the copy of vertex $v_{i'}$ in $T_{j'}$, $u \prec w$ if and only if $(j,i)$ is lexicographically smaller than $(j',i')$. Let $K \subseteq V(T')$ inducing a clique in $T'^{\prec}$.

    Suppose that for every $i \in \intv{1}{n}$, $K$ intersects $\bigcup_{j \in B_i} V(T_j)$ on at most one vertex $u_i$. Then $\{v_i \mid u_i \in K\}$ induces a clique in $T$, and thus $|K| \leq \diomega(T)$.
    Hence, $|K \cap \bigcup_{j \in B_i} V(T_j)| \geq 2$ for some $i \in \intv{1}{n}$. Let $X = K \cap \bigcup_{j \in B_i} V(T_j)$. Let $u, w$ be distinct vertices of $X$ with $u \in V(T_j)$ and $w \in V(T_k)$.

 We cannot have $k > j$, for since $uw \in A(T)$ and $u \prec w$, $K$ would not induce a clique in $T^\prec$. Hence $k = j$, and $\psi_j(u) \neq \psi_j(w)$
 If $K \neq X$, let $y \in K \setminus X$ and let $\ell$ be such that $y \in V(T_\ell)$, then since $K$ induces a clique in $T^\prec$, $\psi_\ell(y) = \psi_j(u)$ and $\psi_\ell(y) = \psi_j(w)$, which contradicts $\psi_j(u) = \psi_j(w)$. 
 Hence, $K = X$. But, since $\prec_{\mid V(T_j)}$ is a $\diomega$-ordering of $T_j$, $|K| \leq \diomega(T)$.
    Thus, $\diomega(T') \leq \diomega(T)$.

    Let us now prove by contradiction that for any $X \subseteq V(T')$, $T'[X]$ contains a copy of $T$ or $T'[\overline{X}]$ contains a copy of $T$.
    Suppose this is not the case.
    Then, for every $i \in \intv{1}{m \times n}$, there exists a vertex $u_i \in X \cap T_i$. For $j \in \intv{1}{n}$, let $C_j = \{\psi_i(u_i) \mid i \in B_j\}$.
    We have that $|C_j| \geq (n-1)^2 + 2$, since if there exists $C \subseteq \intv{1}{n(n-1) + 1} \setminus C_j$ of size $n$, let $i = \varphi^{-1}(C)$, then
    $\psi_i(u_i) \in C$, contradicting that $\psi_i(u_i) \in C_j$. Hence $|\bigcap_{j \in \intv{1}{n}} C_j| \geq 1$, which implies there exists $c \in \bigcap_{j \in \intv{1}{n}} C_j$.
    Thus, for every $j \in \intv{1}{n}$, there exists a vertex $w_j \in T_i \cap X$ with $i \in B_j$ such that $\psi_i(w_j) = c$.
    By definition of $T'$, $w_j w_j' \in A(T')$ if and only if $v_j v_j' \in A(T)$, and thus $T'[\{w_j \mid j \in \intv{1}{n}\}]$ is isomorphic to $T$, yet is included in $X$, a contradiction.

\end{proof}

Using this first lemma, we can then define the tournament $T_k$ as follows:

\begin{definition}\label{def:T_k}
    Let $k \in \mathbb{N}$. Let $T$ be any tournament with $\diomega(T) = k$. We denote as $T_k$ a tournament such that $\diomega(T_k) = k$ and for any $X \subseteq V(T_k)$, $T_k[X]$ or $T_k[\overline{X}]$ contains a copy of $T$.
\end{definition}

Note that such a $T_k$ necessarily exists due to Lemma~\ref{lem:cut}.
Using $T_k$, we can modify constructions to force the order of some pairs of vertices.

\medskip

\begin{lemma}\label{lem:complete}
    Let $k \in \mathbb{N}$. Let $D$ be a digraph, let $u,v \in V(D)$ and let $W \subseteq u^+ \cap v^-$ such that $D[W] = T_k$. If $\prec$ is an ordering of $V(D)$ with $\omega(D^\prec) \leq k$, then $u \prec v$.
\end{lemma}

\begin{proof}

  Suppose $v \prec u$. Let $W^\prec = \{w \in W \mid w \prec u\}$ and $W^\succ = \{w \in W \mid u \prec w\}$. Since $u \Ra W$, $\diomega(W^\prec) \leq k - 1$. Thus, there is a copy of $T$ in $W^\succ$ and thus $\diomega(D[W^\succ]) = k$. But since $W \Ra v$ and $v \prec u \prec W^\succ$, this implies that $\omega(D^\prec) > k$, a contradiction.
\end{proof}

We can also use $T_k$ to leverage $\diomega$-orderings of a tournament of a given clique number into $\diomega$-orderings of a tournament with a larger clique number.

\begin{lemma}\label{lem:uplift}
Let $k \in \mathbb{N}$. If $D$ is a digraph with $\diomega(D) = k-1$, then $\diomega(\Delta(1, D, T_k)) = k$ and $\{\prec \mid \omega(D^\prec) = k-1\} = \{\prec_{\mid V(D)} \mid \omega(\Delta(1, D, T_k)^\prec) = k\}$
\end{lemma}

\begin{proof}

    We will prove that $D' = \Delta(1, D, T_k)$ has the desired property. Clearly $\diomega(D') \geq k$ since $T_k$ is a subdigraph of $D'$. Let $\{v\} = V(D') \setminus V(D) \setminus V(T_k)$.
   
   Let $\prec_k$ be a $\diomega$-ordering of $T_k$. For any $\diomega$-ordering $\prec$ of $V(D)$, let $\prec'$ be the ordering of $V(D')$ obtained by concatenating the vertices of $D$ ordered according to $\prec$, followed by the vertices of $T_k$ ordered according to $\prec_k$ followed by vertex $v$. Then $\prec'$ satisfies $\omega(D'^{\prec'}) = k$, and $\prec'_{\mid V(D)} = \prec$. Thus $\diomega(D') \leq k$ and $\{\prec \mid \omega(D^\prec) = k-1\} \subseteq \{\prec_{\mid V(D)} \mid \omega(\Delta(1, D, T_k)^\prec) = k\}$.

    Let $\prec'$ be an ordering of $V(D')$ with $\omega(D'^{\prec'}) = k$. Let $u \in V(D)$: by Lemma~\ref{lem:complete}, $u \prec' v$.
    Thus, $T_k \prec' v$, which implies that $\omega(T^{\prec'}_k) = k - 1$. Thus $\{\prec_{\mid V(D)} \mid \omega(\Delta(1, D, T_k)^\prec) = k\} \subseteq \{\prec \mid \omega(D^\prec) = k-1\}$.
\end{proof}

\subsection{Gadgets and NP-completeness}

Using the tournaments $T_k$ constructed in Definition~\ref{def:T_k}, we can now build gadgets that will be used to prove that \textsc{$k$-DIOMEGA-ORDERING} is NP-complete. We will do so using a reduction from $3$-SAT.

\medskip

In order to do so, we need to build a tournament in which one fixed edge is forward if and only if another fixed edge is not. This tournament will encode a variable, and which edge is forward will encode whether this variable is set to true.

\begin{lemma}\label{lem:T_var}
    There exists a tournament $T_{var}$ and $uv, wx$ disjoint arcs of $T_{var}$ such that:
\begin{itemize}
    \item $\diomega(T_{var}) = 3$
    \item in any $\diomega$-ordering, exactly one of $uv$ or $wx$ is forward.
    \item there exists a $\diomega$-ordering in which $uv$ is forward
    \item there exists a $\diomega$-ordering in which $uv$ is backward
\end{itemize}
\end{lemma}

\medskip

This proof relies mainly on computer help. However, since checking all orderings of a graph is a complex task, we proceed by first finding such a tournament with $\diomega = 2$, and then leveraging it using Lemma~\ref{lem:uplift}.

\medskip

\begin{proof}

    Let $T$ be the tournament with the following ordered graph as a backedge graph:

\begin{figure}[H]
    \centering
     \begin{tikzpicture}
         \vertex (8) at (1,0) {$1$};
         \vertex (2) at (2,0) {$2$};
         \vertex (1) at (3,0) {$3$};
         \vertex (5) at (4,0) {$4$};
         \vertex (3) at (5,0) {$5$};
         \vertex (4) at (6,0) {$6$};
         \vertex (0) at (7,0) {$7$};
         \vertex (7) at (8,0) {$8$};
         \vertex (6) at (9,0) {$9$};

         \draw[bend right=30] (8) to (0);
         \draw[bend right=30] (8) to (7);
         \draw[bend right=30] (8) to (6);
         \draw[bend right=30] (2) to (3);
         \draw[bend right=30] (2) to (4);
         \draw[bend left=30] (2) to (7);
         \draw[bend left=30] (1) to (7);
         \draw[bend left=30] (1) to (6);
         \draw[bend left=30] (5) to (4);
         \draw[bend left=30] (3) to (6);
         \draw[bend left=30] (4) to (6);
     \end{tikzpicture}

     \caption{The backedge graph corresponding to the ordering $1 \prec 2 \prec 3 \prec 4 \prec 5 \prec 6 \prec 7 \prec 8 \prec 9$}
     \label{fig:not_first}
    \end{figure}
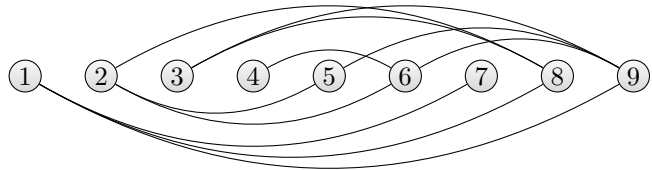
    
    and let $uv = 7 \rightarrow 9$ and $wx = 8 \rightarrow 3$.

\medskip

    Note that $T$ is not transitive and the ordering $1 \prec 2 \prec 3 \prec 4 \prec 5 \prec 6 \prec 7 \prec 8 \prec 9$ is such that $\omega(T^\prec) = 2$, thus $\diomega(T) = 2$.  

    The ordering $1 \prec 2 \prec 3 \prec 4 \prec 5 \prec 6 \prec 7 \prec 8 \prec 9$ is a $\diomega$-ordering in which $uv$ is forward. See Figure~\ref{fig:not_first}.

    The ordering $6 \prec 8 \prec 2 \prec 9 \prec 1 \prec 3 \prec 4 \prec 5 \prec 7$ is a $\diomega$-ordering in which $uv$ is backward. See Figure~\ref{fig:not_second}.

    \begin{figure}[H]
    
    \centering
     \begin{tikzpicture}
    
         \vertex (4) at (1,0) {$6$};
         \vertex (7) at (2,0) {$8$};
         \vertex (2) at (3,0) {$2$};
         \vertex (6) at (4,0) {$9$};
         \vertex (8) at (5,0) {$1$};
         \vertex (1) at (6,0) {$3$};
         \vertex (5) at (7,0) {$4$};
         \vertex (3) at (8,0) {$5$};
         \vertex (0) at (9,0) {$7$};

         \draw[bend right=30] (4) to (6);
         \draw[bend right=30] (4) to (8);
         \draw[bend right=30] (4) to (1);
         \draw[bend right=30] (4) to (3);
         \draw[bend right=30] (7) to (5);
         \draw[bend left=30] (7) to (3);
         \draw[bend left=30] (7) to (0);
         \draw[bend left=30] (2) to (8);
         \draw[bend left=30] (2) to (3);
         \draw[bend left=30] (6) to (5);
         \draw[bend left=30] (6) to (0);
         \draw[bend left=30] (8) to (0);
    \end{tikzpicture}
        \caption{The backedge graph of $T$ corresponding to the ordering $6 \prec 8 \prec 2 \prec 9 \prec 1 \prec 3 \prec 4 \prec 5 \prec 7$}
        \label{fig:not_second}
    \end{figure}
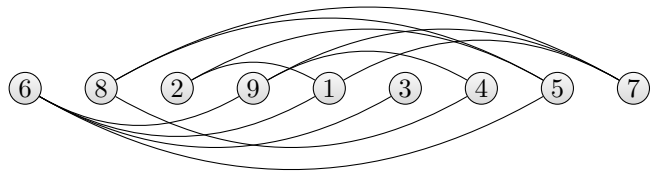
    
It can be checked that in every $\diomega$-ordering, exactly one of $uv$ or $wx$ is forward. See for example the Python code in Figure~\ref{fig:not_code}.

\begin{figure}[H]
    \centering
\begin{python}
import itertools

T = [
  [0,1,1,1,1,1,0,0,0],
  [0,0,1,1,0,0,1,0,1],
  [0,0,0,1,1,1,1,0,0],
  [0,0,0,0,1,0,1,1,1],
  [0,1,0,0,0,1,1,1,0],
  [0,1,0,1,0,0,1,1,0],
  [1,0,0,0,0,0,0,1,1],
  [1,1,1,0,0,0,0,0,1],
  [1,0,1,0,1,1,0,0,0]
]

for P in itertools.permutations(range(9)):
  omega_at_most_two = True
  for i in range(9):
    for j in range(i + 1,9):
      for k in range(j + 1, 9):
        if T[P[j]][P[i]] and T[P[k]][P[j]] and T[P[k]][P[i]]:
          omega_at_most_two = False
  if omega_at_most_two:
    assert((P.index(6) < P.index(8)) != (P.index(7) < P.index(2)))
\end{python}
     \caption{Python code to check that in every $\diomega$-ordering, exactly one of $uv$ or $wx$ is forward}\label{fig:not_code}
    \end{figure}

    By Lemma~\ref{lem:uplift}, there exists a tournament $T_3$ such that $T_{var} = \Delta(1, T, T_3)$ satisfies the desired property.

\end{proof}

We also need a gadget tournament to encode clauses, In order to do so, we build a tournament in which three fixed arcs are such that at least one of them is always backward. This will encode a disjunction.

\medskip

\begin{lemma}\label{lem:T_clause}
    There exists a tournament $T_{clause}$ and $uv, wx, yz$ pairwise disjoint arcs of $T_{clause}$ such that:
\begin{itemize}
    \item $\diomega(T_{clause}) = 3$
    \item for every $\diomega$-ordering $\prec$ of $T_{clause}$, one of $uv, wx, yz$ is backward
    \item for every two arcs in $uv, wx, yz$, there exists a $\diomega$-ordering $\prec$ in which these two arcs are forward

\end{itemize}
\end{lemma}

\medskip

Similarly to Lemma~\ref{lem:T_var}, this proof relies mainly on computer help and we proceed by first finding such a tournament with clique number $2$, then leveraging it thanks to Lemma~\ref{lem:uplift}.

\medskip

\begin{proof}

    Let $T$ be the tournament with the following ordered graph as a backedge graph:

    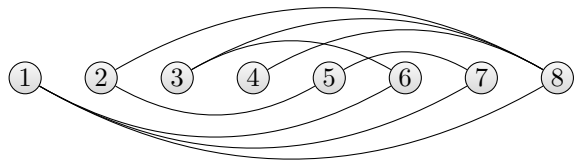
\begin{figure}[H]
    \centering
     \begin{tikzpicture}
    
         \vertex (5) at (1,0) {$1$};
         \vertex (4) at (2,0) {$2$};
         \vertex (0) at (3,0) {$3$};
         \vertex (1) at (4,0) {$4$};
         \vertex (3) at (5,0) {$5$};
         \vertex (7) at (6,0) {$6$};
         \vertex (2) at (7,0) {$7$};
         \vertex (6) at (8,0) {$8$};

         \draw[bend right=30] (5) to (7);
         \draw[bend right=30] (5) to (2);
         \draw[bend right=30] (5) to (6);
         \draw[bend right=30] (4) to (3);
         \draw[bend left=30] (4) to (6);
         \draw[bend left=30] (0) to (7);
         \draw[bend left=30] (0) to (6);
         \draw[bend left=30] (1) to (6);
         \draw[bend left=30] (3) to (2);
     \end{tikzpicture}
     \caption{The backedge graph corresponding to the ordering $1 \prec 2 \prec 3 \prec 4 \prec 5 \prec 6 \prec 7 \prec 8$}
        \label{fig:or_first}
    \end{figure}

    and let $uv = 5 \rightarrow 6$, $wx = 2 \rightarrow 4$ and $yz = 8 \rightarrow 3$.

\medskip

    Note that $T$ is not transitive and the ordering $1 \prec 2 \prec 2 \prec 3 \prec 4 \prec 5 \prec 6 \prec 7 \prec 8$ is such that $\omega(T^\prec) = 2$, thus $\diomega(T) = 2$.

    The ordering $1 \prec 2 \prec 3 \prec 4 \prec 5 \prec 6 \prec 7 \prec 8$ is a $\diomega$-ordering in which $uv$ and $wx$ are forward. See Figure~\ref{fig:or_first}.

    The ordering $4 \prec 7 \prec 5 \prec 8 \prec 1 \prec 2 \prec 6 \prec 3$ is a $\diomega$-ordering in which $uv$ and $yz$ are forward. See Figure~\ref{fig:or_second}.

    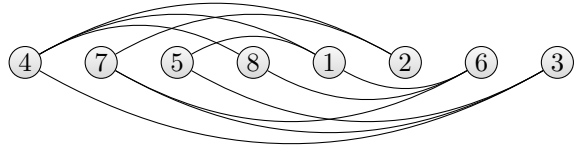
\begin{figure}[H]
    \centering
     \begin{tikzpicture}
    
         \vertex (5) at (5,0) {$1$};
         \vertex (4) at (6,0) {$2$};
         \vertex (0) at (8,0) {$3$};
         \vertex (1) at (1,0) {$4$};
         \vertex (3) at (3,0) {$5$};
         \vertex (7) at (7,0) {$6$};
         \vertex (2) at (2,0) {$7$};
         \vertex (6) at (4,0) {$8$};

         \draw[bend left=30] (1) to (6);
         \draw[bend left=30] (1) to (5);
         \draw[bend right=30] (4) to (1);
         \draw[bend right=30] (1) to (0);
         \draw[bend left=30] (2) to (4);
         \draw[bend left=30] (7) to (2);
         \draw[bend left=30] (0) to (2);
         \draw[bend right=30] (5) to (3);
         \draw[bend right=30] (3) to (0);
         \draw[bend left=30] (7) to (6);
         \draw[bend left=30] (7) to (5);
     \end{tikzpicture}
     \caption{The backedge graph of $T$ corresponding to the ordering $4 \prec 7 \prec 5 \prec 8 \prec 1 \prec 2 \prec 6 \prec 3$}
        \label{fig:or_second}
    \end{figure}
    
    The ordering $1 \prec 2 \prec 4 \prec 6 \prec 7 \prec 5 \prec 8 \prec 3$  is a $\diomega$-ordering in which $wx$ and $yz$ are forward. See Figure~\ref{fig:or_third}.

    \begin{figure}[H]
    \centering
     \begin{tikzpicture}
    
         \vertex (5) at (1,0) {$1$};
         \vertex (4) at (2,0) {$2$};
         \vertex (0) at (8,0) {$3$};
         \vertex (1) at (3,0) {$4$};
         \vertex (3) at (6,0) {$5$};
         \vertex (7) at (4,0) {$6$};
         \vertex (2) at (5,0) {$7$};
         \vertex (6) at (7,0) {$8$};

         \draw[bend right=30] (5) to (7);
         \draw[bend right=30] (5) to (2);
         \draw[bend right=30] (5) to (6);
         \draw[bend right=30] (4) to (3);
         \draw[bend left=30] (4) to (6);
         \draw[bend left=30] (7) to (1);
         \draw[bend left=30] (7) to (3);
         \draw[bend left=30] (1) to (6);
         \draw[bend left=30] (1) to (0);
         \draw[bend left=30] (2) to (0);
         \draw[bend left=30] (3) to (0);
     \end{tikzpicture}
     \caption{The backedge graph of $T$ corresponding to the ordering $1 \prec 2 \prec 4 \prec 6 \prec 7 \prec 5 \prec 8 \prec 3$}
        \label{fig:or_third}

    \end{figure}
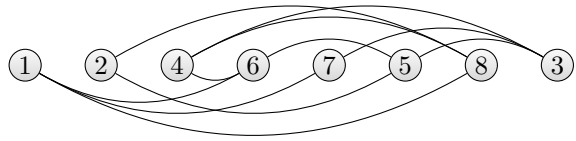

    It can be checked that in every $\diomega$-ordering, at least one of $uv$, $wx$ or $yz$ is backward.
See the Python code in Figure~\ref{fig:or_code}.

\begin{figure}[H]
    \centering
    \begin{python}
import itertools

T = [
  [0,1,1,1,1,0,0,0],
  [0,0,1,1,0,1,1,0],
  [0,0,0,1,1,0,1,0],
  [0,0,0,0,1,1,1,0],
  [0,1,0,0,0,1,0,1],
  [1,0,1,0,0,0,1,1],
  [1,0,0,0,1,0,0,1],
  [1,1,1,1,0,0,0,0],
]

for P in itertools.permutations(range(8)):
  omega_at_most_two = True
  for i in range(8):
    for j in range(i + 1,8):
      for k in range(j + 1, 8):
        if T[P[j]][P[i]] and T[P[k]][P[j]] and T[P[k]][P[i]]:
          omega_at_most_two = False
  if omega_at_most_two:
    assert(P.index(5) < P.index(4) or P.index(3) < P.index(1) or P.index(2) < P.index(7))
\end{python}
\caption{Python code to check that in every $\diomega$-ordering, one of $uv, wx$ or $yz$ is backward}
    \label{fig:or_code}
\end{figure}

    By Lemma~\ref{lem:uplift}, there exists a tournament $T_3$ such that $T_{clause} = \Delta(1, T, T_3)$ satisfies the desired property.

\end{proof}

We can then conclude:

\medskip

\begin{theorem}
\textsc{$3$-DIOMEGA-ORDERING} is NP-complete.
\end{theorem}

\begin{proof}
    Clearly, this problem is in NP. Let us now prove it is NP-hard.

    \medskip

    We will reduce $3$-SAT to this problem.
    Let $\varphi$ be a $3$-SAT formula on variables $v_i$ for $i \in \intv{1}{n}$ with clauses $C_j = C_{j,1} \lor C_{j,2} \lor C_{j,3}$ where $C_{j,k}$ are literals.
    Without loss of generality, we can suppose no variable appears more than once in the literals of a clause.
For every variable $v_i$, we create a copy $A_i$ of $T_{var}$ (as defined in Lemma~\ref{lem:T_var}). Let $f^+_i$ and $f^-_i$ be the respective copies of $uv$ and $wx$ in $A_i$.
For every clause $C_j$, we create a copy $B_j$ of $T_{clause}$ (as defined in Lemma~\ref{lem:T_clause}). Let $e^1_{j}, e^2_{j}, e^{3}_j$ be the respective copies of $uv, wx, yz$ in $B_j$.

Let $T_3$ be defined as in Definition~\ref{def:T_k}.

We will consider the tournament $T = A_1 \Ra \dots \Ra A_n \Ra T_3 \Ra B_1 \Ra \dots \Ra B_m$ in which we revert some arcs as follows: for every literal $C_{j,k}$ corresponding to variable $v_i$
, let $ab = f^+_i$ if it is a positive literal and $ab = f^-_i$ otherwise and let $cd = e^{k}_{j}$. We revert arcs $ac, ad, bc$ and $bd$ (thus after the reversal, $ca, da, cb, db \in A(T)$).

Let us prove that $\diomega(T) = 3$ if and only if $\varphi$ is satisfiable.

    First, suppose $\diomega(T) = 3$ and let $\prec$ be a $\diomega$-ordering of $T$.
    Note that by Lemma~\ref{lem:complete}, for $i \in \intv{1}{n}$ and $j \in \intv{1}{m}$, $V(A_i) \prec V(B_j)$.
    Let $\nu$ be the assignment such that $\nu(v_i)$ if and only if the arc $f^+_i$ is forward or equivalently, if and only if the arc $f^-_i$ is backward.
    We will prove that $\nu$ is a satisfies $\varphi$.
    Let $C_j$ be a clause.
    Since $B_j$ is a copy of $T_{clause}$, there exists $k \in \{1,2,3\}$ such that $e^k_{j} = cd$ is backward, due to Lemma~\ref{lem:T_clause}.
    Let $ab = f^+_i$ if $C^k_j$ is a positive literal and $ab = f^-_i$ otherwise.
    Then, $ab$ must be forward, for otherwise $\{a,b,c,d\}$ would induce a $K_4$ in $T^\prec$ since $ca, da, cb$ and $db$ are backward.
    Thus the literal $C^k_j$ is satisfied. Thus every clause is satisfied by $\nu$. Hence $\varphi$ is satisfiable.

    Now, suppose $\varphi$ is satisfiable. We will build an ordering $\prec$ of $V(T)$ such that $\omega(T^\prec) = 3$. Let $\nu$ be an assignment satisfying $\varphi$, and for $j \in \intv{1}{m}$, let $k_j$ be such that $C_{j,k_{j}}$ is satisfied. Our ordering $\prec$ will be defined as follows:
    \begin{itemize}
        \item First come vertices of $A_1$, ordered following a $\diomega$-ordering of $V(A_1)$ in which $f^+_1$ is forward if and only if $\nu(v_1)$
        \item \dots
        \item then come vertices of $A_n$, ordered following a $\diomega$-ordering of $V(A_n)$ in which $f^+_n$ is forward if and only if $\nu(v_n)$
        \item then come vertices of $T_3$, ordered following a $\diomega$-ordering of $T_3$
        \item then come vertices of $B_1$, ordered following a $\diomega$-ordering of $V(B_1)$ in which, for $k \in \{1, 2, 3\} \setminus \{k_1\}$, $e^k_1$ is forward
        \item \dots
        \item then come vertices of $B_m$, ordered following a $\diomega$-ordering of $V(B_m)$ in which, for $k \in \{1, 2, 3\} \setminus \{k_m\}$, $e^k_m$ is forward
    \end{itemize}

    Suppose there exists $K \subset V(T)$ such that $T^\prec[K] = K_4$. $K$ cannot intersect $V(T_3)$, since $V(T_3)$ is anticomplete to $V(T) \setminus V(T_3)$ in $T^\prec$ and $\omega(T^\prec_3) = 3$.
    $K$ cannot intersect both $V(A_i)$ and $V(A_{i'})$ whenever $i \neq i'$ since $A_i$ is anticomplete with $A_{i'}$ in $T^\prec$.
    For the same reason, $K$ cannot intersect both $V(B_j)$ and $V(B_{j'})$ whenever $j \neq j'$.
    Thus, $K \subset V(A_i) \cup V(B_j)$ for some integers $i \in \intv{1}{n}$ and $j \in \intv{1}{m}$.
    Since $\omega(A^\prec_i) = 3$, $K$ intersect $B_j$, and since $\omega(B^\prec_j) = 3$, $K$ intersects $A_i$.
    Since vertices of $A_i$ have at most $2$ inneighbours in $V(B_j)$ and $V(A_i) \prec V(B_j)$, $K$ intersects $B_j$ on at most $2$ vertices.
    Similarly, since vertices of $B_j$ have at most $2$ outneighbours in $V(A_i)$ and $V(A_i) \prec V(B_j)$, $K$ intersect $A_i$ on at most $2$ vertices.
    Thus $K = \{u,v,w,x\}$ with $v \prec u$ two distinct vertices of $A_i$ and $x \prec w$ two distinct vertices of $B_j$.
    This is only possible if there exists $k \in \{1,2,3\}$ such that $wx = e^k_j$ and $uv = f^+_i$ if $C^k_j$ is a positive literal and $uv = f^-_i$ otherwise.
    Since $e^k_j$ is backward, this implies $k = k_j$. and thus $uv$ is not backward. Thus, $T^\prec[K] \neq K_4$, a contradiction. 

\end{proof}

Using Lemma~\ref{lem:uplift}, we directly get:

\begin{corollary}
    For any fixed $k \geq 3$, \textsc{$k$-DIOMEGA-ORDERING} is NP-complete.
\end{corollary}

\section{A counterexample to a conjecture of Aboulker, Aubian, Charbit and Lopes}

In \cite{AACL23}, Aboulker, Aubian, Charbit and Lopes proposed the following conjecture, a tournament version of \Gya-Sumner conjecture:

\begin{conjecture}\label{conj:gyarfas_sumner}[Aboulker, Aubian, Charbit, Lopes, 2023]
    Let $T$ be a tournament whose one backedge graph is a forest. Let $k \in \mathbb{N}$. Tournaments not containing $T$ as a subtournament and with clique number at most $k$ have bounded dichromatic number.
\end{conjecture}

The main result of this section is a disproof of this conjecture:

\begin{theorem}
    Conjecture~\ref{conj:gyarfas_sumner} is false.
\end{theorem}

In order to do so, we adapt the construction from Lemma~\ref{lem:cut}.

\begin{definition}\label{def:pi}
    Let $T$ be a tournament on $n$ vertices. We define the tournament $\Pi(T)$ and the ordering $\prec^{\Pi}$ of $V(\Pi(T))$ as follows:
    
    \medskip

    Let $m = {{2n-1} \choose n}$. We start from $2m + 1$ copies of $T$, denoted as $A_1, \dots, A_m, B, C_1, \dots, C_m$.
    Let $\varphi \colon \intv{1}{m} \to {{\intv{1}{2n+1}} \choose n}$ be a bijection and, for $i \in \intv{1}{m}$, let $\psi_j \colon V(T) \to \varphi(j)$ be a bijection.
    $\Pi(T)$ is defined from $A_1 \Ra \dots \Ra A_m \Ra B \Ra C_1 \Ra \dots \Ra C_m$ by reversing arc $uv$ with $u \in V(A_i), v \in V(C_j)$ if and only if $\psi_i(u) = \psi_j(v)$.
    $\prec^{\Pi}$ is defined as the concatenation of any $\diomega$-orderings of $A_1$, then $A_2$, \dots, then $A_m$, then $B$, then $C_1$, \dots and then $C_m$.
\end{definition}

An easy consequence of Lemma~\ref{lem:cut} is the following lemma:
\begin{lemma}
    Let $T$ be a tournament. If $\diomega(T) = 2$ then $\omega(\Pi(T)^{\prec^{\Pi}}) = 2$.
\end{lemma}

\begin{proof}

    Since $T$ contains a directed triangle, $\Pi(T)$ is included in the construction of $T'$ as defined in the proof of Lemma~\ref{lem:cut}, which by Lemma~\ref{lem:cut} has the same clique number as $T$. Since $\Pi(T)$ contains $T$, $\omega(\Pi(T)^{\prec^\Pi}) = 2$.
\end{proof}

Thus $\Pi(T)$ preserves $\diomega$.

\begin{lemma}
    Let $T$ be a tournament. $\dic(\Pi(T)) > \dic(T)$.
\end{lemma}

\begin{proof}

    Suppose this is not the case, and let $X$ be an acyclic subset of $\Pi(T)$ such that $\dic(\Pi(T)[\overline{X}]) < \dic(T)$.
    Thus $X$ must intersect $A_1, \dots, A_m, B, C_1, \dots, C_m$ on at least one vertex $a_1, \dots, a_m, b, c_1, \dots, c_m$.
    Let $A = \{\psi_i(a_i) \mid i \in \intv{1}{m}\}$ and $C = \{\psi_i(c_i) \mid i \in \intv{1}{m}\}$.
    $|A| \geq n$ since every subset $S \in {\intv{1}{2n - 1} \choose n}$ intersects $X$ on vertex $a_{\varphi^{-1}(S)}$.
    Similarly, $|C| \geq n$, and thus $|A \cap C| \geq 1$.
    Let $x \in A \cap C$, and let $i, j \in \intv{1}{m}$ be such that $a_i \in X \cap V(A_i)$, $c_j \in X \cap V(C_j)$ and $\psi_i(a_i) = \psi_j(c_j) = x$.
    Then $a_i \rightarrow b \rightarrow c_j \rightarrow a_i$ is a directed triangle in $C$, a contradiction.
\end{proof}

Our counterexample will then be defined recursively as $D_1 = C_3$ and $D_k = \Pi(D_{k-1})$ for $k \geq 2$. We directly get the following result as a corollary from the previous two theorems:

\begin{corollary}
    $\dic(\{D_k \mid k \geq 1\})$ is unbounded, but for any $k \geq 1$, $\diomega(D_k) = 2$.
\end{corollary}

We will prove that there exists a tournament with a backedge graph that is a forest which does not appear as a subgraph of any $D_k$.
In order to do so, we need this sufficient condition on subtournaments of $\{D_k \mid k \geq 1\}$.

\begin{lemma}\label{lem:rules}
    Let $T$ be a strong tournament that is a subgraph of $D_k$ for some $k \geq 1$. Then $T$ admits a $\diomega$-ordering $\prec$ and a vertex $x$ satisfying the following four rules:
    \begin{enumerate}
        \item there exists $y \prec x$
        \item for any $a, b, c, d$ with $a, b \prec x \preceq c, d$, if $ca,da,cb \in A(T)$ then $db \in A(T)$
        \item for any $a \preceq u \prec w \preceq b \prec x \preceq v$, if there is an  $ab$-path in $T^\prec[\{y \mid y \prec x\}]$, either $uv \in A(T)$ or $wv \in A(T)$
        \item for any $v \prec x \preceq a \preceq u \prec w \preceq b$ if there is an $ab$-path in $T^\prec[\{y \mid x \preceq y\}]$, either $vu \in A(T)$ or $vw \in A(T)$
    \end{enumerate}
\end{lemma}

\begin{proof}

    Let $k$ be minimum such that there exists $X$ with $D_k[X] = T$.
    It is straightforward to check the lemma whenever $k = 1$, and thus we can suppose $k \geq 2$.
    Let $m, n, A_1, \dots, A_m, B, C_1, \dots, C_m$ and $\prec^{\Pi}$ be defined as in Definition~\ref{def:pi}.
    Let $\prec = \prec^\Pi_{\mid X}$.
    $X$ must intersect one of $(C_i)_{i \in \intv{1}{m}}$ for $T$ is strong and $T$ is not a subtournament of any $A_i$ nor of $B$.
    Let $x$ be the minimum vertex (relative to order $\prec$) in $X \cap \bigcup_{i \in \intv{1}{m}} V(C_i)$.
    We will prove that $\prec$ and $x$ satisfy the four rules.
    In order to do so, we will first need the following claim:

    \begin{claim}\label{clm:wxy}
        Let $w, y \in X$ such that $w \prec x \preceq y$ and $yz \in A(D_k)$. There exist $i \in \intv{1}{m}$ such that $w \in V(A_i)$ and  $j \in \intv{1}{m}$ such that $y \in V(C_j)$.
    \end{claim}

    \begin{proofclaim}

        Since $w \prec x$ and $x$ is the leftmost vertex in $X \cap \bigcup_{i \in \intv{1}{m}} V(C_i)$, either $w \in V(B)$ or $w \in \bigcup_{i \in \intv{1}{m}} V(A_i)$.
        Since $x \preceq y$, there exists $j \in \intv{1}{m}$ such that $y \in V(C_j)$.
        This, combined with the fact that $yw \in A(D_k)$ implies that $w \notin V(B)$.
        Thus, there exists $i \in \intv{1}{m}$ such that $w \in V(A_i)$.

    \end{proofclaim}

    \medskip

    \textbf{Rule 1:}

    If there does not exist any $y \prec x$, then either $T$ is not strong or $T$ is a subtournament of some $C_i$, a contradiction.

    \medskip

    \textbf{Rule 2:}

    Let  $a, b, c, d \in X$ with $a, b \prec x \preceq c, d$ and $ca,da,cb \in A(D_k)$.
    By Claim~\ref{clm:wxy}, there exists $i, i', j, j' \in \intv{1}{m}$ such that $a \in V(A_i), b \in V_{i'}, c \in V(C_j)$ and $d \in V(C_{j'})$.
    Then, the fact that $ca, da, cb \in A(D_k)$ implies that $\psi_i(a) = \psi_j(c)$, $\psi_i(a) = \psi_{j'}(d)$ and $\psi_j(c) = \psi_{i'}(b)$.
    Hence $\psi_{i'}(b) = \psi_{j'}(d)$, and thus $db \in A(D_k)$.

    \medskip

    \textbf{Rule 3:}
    
    Suppose, by contradiction, that there exist $a,b,u,v,w$ such that there is an  $ab$-path in $T^\prec[\{y \mid y \prec x\}]$
            and $a \preceq u \prec w \preceq b \prec x \preceq v$,  $vu \in A(D_k)$ and $vw \in A(D_k)$.
    By Claim~\ref{clm:wxy}, there exists $i'$ such that $u \in V(A_{i'})$ and $j$ such that $v \in V(C_j)$.
    Since $a \preceq u$, there exists $i$ such that $a \in V(A_i)$, and since there exists an $ab$-path in $T^\prec[\{y \mid y \prec x\}]$, $b \in V(A_i)$.
    But vertices of $A_i$ are continguous in the ordering $\prec$, thus $u, w \in V(A_i)$.
    Since $vu \in A(D_k)$, $\psi_i(u) = \psi_j(v)$.
    Since $vw \in A(D_k)$, $\psi_i(w) = \psi_j(v)$.
    Thus $\psi_i(u) = \psi_i(w)$, a contradiction to the fact that $\psi_i$ is a bijection.

    \medskip

    \textbf{Rule 4:}

        Suppose, by contradiction, that there exist $a,b,u,v,w$ such that there is an $ab$-path in $T^\prec[\{y \mid x \preceq y\}]$, $v \prec x \preceq a \preceq u \prec w \preceq b$, $uv \in A(D_k)$ and $wv \in A(D_k)$.
    By Claim~\ref{clm:wxy}, there exists $i$ such that $v \in V(A_{i})$ and $j'$ such that $w \in V(C_{j'})$.
    Since $w \preceq b$, there there exists $j$ such that $b \in V(C_j)$, and since there exists an $ab$-path in $T^\prec[\{y \mid x \preceq y\}]$, $a \in V(C_j)$.
    But vertices of $C_j$ are continguous in the ordering $\prec$, thus $u, w \in V(C_j)$.
    Since $uv \in A(D_k)$, $\psi_i(v) = \psi_j(u)$.
    Since $vw \in A(D_k)$, $\psi_i(v) = \psi_j(w)$.
    Thus $\psi_j(u) = \psi_j(w)$, a contradiction to the fact that $\psi_j$ is a bijection.
\end{proof}

    Let $R_5$ be the tournament of Figure~\ref{fig:circulant}:

    \begin{figure}[H]
     \centering
     \begin{tikzpicture}[line cap=round,line join=round,>=triangle 45, scale=1.2]

         \vertex (1) at (1,0) {$1$};
         \vertex (2) at (0.309,0.951) {$2$};
         \vertex (3) at (-0.809,0.588) {$3$};
         \vertex (4) at (-0.809,-0.588) {$4$};
         \vertex (5) at (0.309,-0.951) {$5$};

         \draw[->-, bend right=30] (1) to (2);
         \draw[->-] (1) to (3);
         \draw[->-, bend right=30] (2) to (3);
         \draw[->-] (2) to (4);
         \draw[->-, bend right=30] (3) to (4);
         \draw[->-] (3) to (5);
         \draw[->-, bend right=30] (4) to (5);
         \draw[->-] (4) to (1);
         \draw[->-, bend right=30] (5) to (1);
         \draw[->-] (5) to (2);

     \end{tikzpicture}
        \caption{$R_5$}\label{fig:circulant}
    \end{figure}

    We can now conclude this section with our main theorem.

\begin{theorem}
    Conjecture~\ref{conj:gyarfas_sumner} is false.
\end{theorem}

\begin{proof}

    We will proceed by proving that $R_5$ has a backedge graph that is a forest, yet is not a subtournament of any $D_k$, for $k \geq 1$.
    In order to do so, let us first enumerate all backedge graphs of $R_5$ using the Python code in Figure~\ref{fig:cex_code} (note that we restrict ourselves to $\diomega$-orderings in which vertex $1$ stays at first position, since $R_5$ is vertex-transitive).

\begin{figure}[H]
    \centering
\begin{python}
import itertools

T = [
  [0,1,1,0,0],
  [0,0,1,1,0],
  [0,0,0,1,1],
  [1,0,0,0,1],
  [1,1,0,0,0]
]

for P in itertools.permutations(range(5)):
  omega_at_most_two = True
  for i in range(5):
    for j in range(i + 1,5):
      for k in range(j + 1, 5):
        if T[P[j]][P[i]] and T[P[k]][P[j]] and T[P[k]][P[i]]:
          omega_at_most_two = False
  if P[0] == 0 and omega_at_most_two:
    print(P)
\end{python}
    \caption{A Python code enumerating all $\diomega$-orderings of $R_5$.}\label{fig:cex_code}
\end{figure}

    \begin{figure}[H]
        \centering
        \begin{tikzpicture}[scale=1.1]

\begin{scope}[xshift=3cm,yshift=0cm]
 \vertex (0) at (0,0) {$1$};
 \vertex (1) at (1,0) {$2$};
 \vertex (2) at (2,0) {$3$};
 \vertex (3) at (3,0) {$4$};
 \vertex (4) at (4,0) {$5$};

\draw[bend right=30] (0) to (3);
\draw[bend left=30] (0) to (4);
\draw[bend right=30] (1) to (4);
\end{scope}

\begin{scope}[xshift=0cm,yshift=3cm]
 \vertex (0) at (0,0) {$1$};
 \vertex (1) at (1,0) {$2$};
 \vertex (3) at (2,0) {$4$};
 \vertex (2) at (3,0) {$3$};
 \vertex (4) at (4,0) {$5$};

\draw[bend left=30] (0) to (3);
\draw[bend left=30] (0) to (4);
\draw[bend right=30] (1) to (4);
\draw[bend right=30] (3) to (2);
\end{scope}

\begin{scope}[xshift=6cm,yshift=3cm]
 \vertex (0) at (0,0) {$1$};
 \vertex (1) at (1,0) {$2$};
 \vertex (3) at (2,0) {$4$};
 \vertex (4) at (3,0) {$5$};
 \vertex (2) at (4,0) {$3$};

\draw[bend left=30] (0) to (3);
\draw[bend right=30] (0) to (4);
\draw[bend left=30] (1) to (4);
\draw[bend left=30] (3) to (2);
\draw[bend right=30] (4) to (2);
\end{scope}

\begin{scope}[xshift=0cm,yshift=6cm]
 \vertex (0) at (0,0) {$1$};
 \vertex (2) at (1,0) {$3$};
 \vertex (1) at (2,0) {$2$};
 \vertex (3) at (3,0) {$4$};
 \vertex (4) at (4,0) {$5$};

\draw[bend right=30] (0) to (3);
\draw[bend left=30] (0) to (4);
\draw[bend right=30] (2) to (1);
\draw[bend left=30] (1) to (4);
\end{scope}

\begin{scope}[xshift=6cm,yshift=6cm]
 \vertex (0) at (0,0) {$1$};
 \vertex (2) at (1,0) {$3$};
 \vertex (3) at (2,0) {$4$};
 \vertex (1) at (3,0) {$2$};
 \vertex (4) at (4,0) {$5$};

\draw[bend left=30] (0) to (3);
\draw[bend left=30] (0) to (4);
\draw[bend left=30] (2) to (1);
\draw[bend right=30] (3) to (1);
\draw[bend right=30] (1) to (4);
\end{scope}

\begin{scope}[xshift=0cm,yshift=9cm]
 \vertex (0) at (0,0) {$1$};
 \vertex (2) at (1,0) {$3$};
 \vertex (3) at (2,0) {$4$};
 \vertex (4) at (3,0) {$5$};
 \vertex (1) at (4,0) {$2$};

\draw[bend left=30] (0) to (3);
\draw[bend right=30] (0) to (4);
\draw[bend right=30] (2) to (1);
\draw[bend left=30] (3) to (1);
\end{scope}

\begin{scope}[xshift=6cm,yshift=9cm]
 \vertex (0) at (0,0) {$1$};
 \vertex (3) at (1,0) {$4$};
 \vertex (1) at (2,0) {$2$};
 \vertex (2) at (3,0) {$3$};
 \vertex (4) at (4,0) {$5$};

\draw[bend right=30] (0) to (3);
\draw[bend left=30] (0) to (4);
\draw[bend right=30] (3) to (1);
\draw[bend left=30] (3) to (2);
\draw[bend left=30] (1) to (4);
\end{scope}

\begin{scope}[xshift=0cm,yshift=12cm]
 \vertex (0) at (0,0) {$1$};
 \vertex (3) at (1,0) {$4$};
 \vertex (1) at (2,0) {$2$};
 \vertex (4) at (3,0) {$5$};
 \vertex (2) at (4,0) {$3$};

\draw[bend right=30] (0) to (3);
\draw[bend left=30] (0) to (4);
\draw[bend right=30] (3) to (1);
\draw[bend left=30] (3) to (2);
\draw[bend right=30] (1) to (4);
\draw[bend right=30] (4) to (2);
\end{scope}

\begin{scope}[xshift=6cm,yshift=12cm]
 \vertex (0) at (0,0) {$1$};
 \vertex (3) at (1,0) {$4$};
 \vertex (4) at (2,0) {$5$};
 \vertex (1) at (3,0) {$2$};
 \vertex (2) at (4,0) {$3$};

\draw[bend right=30] (0) to (3);
\draw[bend left=30] (0) to (4);
\draw[bend left=30] (3) to (1);
\draw[bend right=30] (3) to (2);
\draw[bend left=30] (4) to (2);
\end{scope}

        \end{tikzpicture}
        \caption{All $\diomega$-orderings of $R_5$}\label{fig:all_diomega}
    \end{figure}

\bigskip

The list of all such $\diomega$-orderings can be found in Figure~\ref{fig:all_diomega}. Note that the backedge graph corresponding to the order $1 \prec 2 \prec 3 \prec 4 \prec 5$ is a forest.

\begin{figure}[H]
    \begin{center}
        \begin{tabular}{ |c c c| }

\hline
$\diomega$-ordering $\prec$ & $x$ & broken rule\\ 
\hline
$1 \prec 2 \prec 3 \prec 4 \prec 5$ & 1 & rule $1$\\
& 2 & rule $4$ with $a = 2, b = 5, u = 4, v = 1$ and $w = 5$\\
& 3 & rule $2$ with $a = 1, b = 2, c = 5 $ and $d = 4$\\
& 4 & rule $2$ with $a = 1, b = 2, c = 5 $ and $d = 4$\\
& 5 & rule $3$ with $a = 1, b = 4, u = 1, v = 5$ and $w = 2$\\
\hline
$1 \prec 2 \prec 4 \prec 3 \prec 5$ & 1 & rule $1$\\
& 2 & rule $4$ with $a = 2, b = 5, u = 4, v = 1$ and $w = 5$\\
& 3 & rule $3$ with $a = 1, b = 4, u = 1, v = 5$ and $w = 2$\\
& 4 & rule $2$ with $a = 1, b = 2, c = 5 $ and $d = 4$\\
& 5 & rule $3$ with $a = 1, b = 4, u = 1, v = 5$ and $w = 2$\\
\hline
$1 \prec 2 \prec 4 \prec 5 \prec 3$ & 1 & rule $1$\\
& 2 & rule $4$ with $a = 2, b = 5, u = 4, v = 1$ and $w = 5$\\
& 3 & rule $3$ with $a = 1, b = 5, u = 4, v = 3$ and $w = 5$\\
& 4 & rule $2$ with $a = 1, b = 2, c = 5 $ and $d = 4$\\
& 5 & rule $3$ with $a = 1, b = 4, u = 1, v = 5$ and $w = 2$\\
\hline
$1 \prec 3 \prec 2 \prec 4 \prec 5$ & 1 & rule $1$\\
& 2 & rule $4$ with $a = 2, b = 5, u = 4, v = 1$ and $w = 5$\\
& 3 & rule $4$ with $a = 3, b = 5, u = 4, v = 1$ and $w = 5$\\
& 4 & rule $2$ with $a = 1, b = 2, c = 5 $ and $d = 4$\\
& 5 & rule $3$ with $a = 1, b = 4, u = 1, v = 5$ and $w = 2$\\
\hline
$1 \prec 3 \prec 4 \prec 2 \prec 5$ & 1 & rule $1$\\
& 2 & rule $3$ with $a = 1, b = 4, u = 3, v = 2$ and $w = 4$\\
& 3 & rule $4$ with $a = 3, b = 5, u = 4, v = 1$ and $w = 5$\\
& 4 & rule $4$ with $a = 4, b = 5, u = 4, v = 1$ and $w = 5$\\
& 5 & rule $3$ with $a = 1, b = 2, u = 1, v = 5$ and $w = 2$\\
\hline
$1 \prec 3 \prec 4 \prec 5 \prec 2$ & 1 & rule $1$\\
& 2 & rule $3$ with $a = 1, b = 4, u = 3, v = 2$ and $w = 4$\\
& 3 & rule $4$ with $a = 3, b = 2, u = 4, v = 1$ and $w = 5$\\
& 4 & rule $4$ with $a = 4, b = 2, u = 4, v = 1$ and $w = 5$\\
& 5 & rule $3$ with $a = 1, b = 4, u = 3, v = 2$ and $w = 4$\\
\hline
$1 \prec 4 \prec 2 \prec 3 \prec 5$ & 1 & rule $1$\\
& 2 & rule $4$ with $a = 2, b = 5, u = 2, v = 4$ and $w = 3$\\
& 3 & rule $3$ with $a = 1, b = 2, u = 1, v = 5$ and $w = 2$\\
& 4 & rule $4$ with $a = 4, b = 5, u = 4, v = 1$ and $w = 5$\\
& 5 & rule $3$ with $a = 1, b = 2, u = 1, v = 5$ and $w = 2$\\
\hline
$1 \prec 4 \prec 2 \prec 5 \prec 3$ & 1 & rule $1$\\
& 2 & rule $4$ with $a = 2, b = 3, u = 2, v = 4$ and $w = 3$\\
& 3 & rule $3$ with $a = 1, b = 5, u = 4, v = 3$ and $w = 5$\\
& 4 & rule $4$ with $a = 4, b = 5, u = 4, v = 1$ and $w = 5$\\
& 5 & rule $3$ with $a = 1, b = 2, u = 1, v = 5$ and $w = 2$\\
\hline
$1 \prec 4 \prec 5 \prec 2 \prec 3$ & 1 & rule $1$\\
& 2 & rule $2$ with $a = 4, b = 5, c = 3 $ and $d = 2$\\
& 3 & rule $3$ with $a = 1, b = 5, u = 4, v = 3$ and $w = 5$\\
& 4 & rule $4$ with $a = 4, b = 2, u = 4, v = 1$ and $w = 5$\\
& 5 & rule $4$ with $a = 5, b = 3, u = 2, v = 4$ and $w = 3$\\
\hline

        \end{tabular}
    \end{center}
    \caption{For every $\diomega$-ordering of $R_5$ and every choice of vertex $x$, which rule of Lemma~\ref{lem:rules} is broken?}\label{fig:big_table}
\end{figure}

\bigskip

The table in Figure~\ref{fig:big_table} returns, for each $\diomega$-ordering and each choice of vertex $x$, a rule of Lemma~\ref{lem:rules} that is broken.
By Lemma~\ref{lem:rules}, $R_5$ is never a subtournament of $D_k$ for $k \geq 1$. Thus there exists a tournament with a backedge graph that is a forest such that there are tournaments of clique number $2$ and arbitrarily large dichromatic number not containing it, contradicting Conjecture~\ref{conj:gyarfas_sumner}. 
\end{proof}

\section{Conclusion and future works}

The main problem that remains open is the following conjecture:

\begin{conjecture}
\textsc{$2$-DIOMEGA-ORDERING} is NP-complete.
\end{conjecture}

A way of attacking this problem could be to see it as a problem on strings. For example, let us consider the following problem:

\bigskip

\textsc{PERMUTATION-AVOIDING-SMALL-SUBWORDS}

\textbf{Input:} A finite set $\Sigma$, and $S \subseteq \Sigma^{\leq 3}$.

\textbf{Output:} Does there exist a permutation $P$ of $\Sigma$ such that no $w \in S$ is a subsequence of $P$?

\bigskip

Given a tournament $T$, one can consider the instance with $\Sigma = V(T)$ and $S = \{wvu \mid uv, vw, uw \in A(T)\}$. Then, a permutations $P$ of $\Sigma$ such that no $w \in S$ is a subsequence of $P$ are exactly orderings $\prec$ of $V(T)$ with $\omega(T^\prec) = 2$. Thus if \textsc{PERMUTATION-AVOIDING-SMALL-SUBWORDS} was solvable in polynomial time, so would be \textsc{$2$-DIOMEGA-ORDERING}.

Unfortunately, one can prove that \textsc{PERMUTATION-AVOIDING-SMALL-SUBWORDS}. However, one could study slight variations of this problem: for example, we can assume that if $abc, dbe \in S$ then $abe, dbc \in S$.

Another problem that is mentioned by Nguyen, Scott and Seymour in \cite{NSS23} is the question of whether \textsc{$k$-DIOMEGA-ORDERING} is in co-NP. This remains open.

\bigskip

While Conjecture~\ref{conj:gyarfas_sumner} is now disproved, one could ask which tournaments are such that the class of tournaments not containing them is $\dic$-bounded: we could not find other counterexamples on $5$ vertices.

\subsubsection*{Acknowledgements}

I would like to thank Pierre Aboulker and Pierre Charbit for the precious discussion that led to considering whether I could modify the construction of Lemma~\ref{lem:cut} to disprove Conjecture~\ref{conj:gyarfas_sumner}. I would also like to thank Louis Jachiet for the idea of the proof that \textsc{PERMUTATION-AVOIDING-SMALL-SUBWORDS} is NP-complete.

\bibliographystyle{plain}
\bibliography{ms} 

\begin{thebibliography}{10}

\bibitem{AACL23}
Pierre Aboulker, Guillaume Aubian, Pierre Charbit, and Raul Lopes.
\newblock Clique number of tournaments, 2023.

\bibitem{BG18}
J\text{\o}rgen Bang-Jensen and Gregory Gutin.
\newblock {\em Classes of Directed Graphs}.
\newblock Springer Publishing Company, Incorporated, 1$^{st}$ edition, 2018.

\bibitem{BM08}
J.A. Bondy and U.S.R Murty.
\newblock {\em Graph Theory}.
\newblock Springer Publishing Company, Incorporated, 1st edition, 2008.

\bibitem{DS86}
William H.~E. Day and David Sankoff.
\newblock Computational complexity of inferring phylogenies by compatibility.
\newblock {\em Systematic Zoology}, 35(2):224--229, 1986.

\bibitem{D05}
Reinhard Diestel.
\newblock {\em Graph Theory (Graduate Texts in Mathematics)}.
\newblock Springer, 2005.

\bibitem{E59}
P.~Erdös.
\newblock Graph theory and probability.
\newblock {\em Canadian Journal of Mathematics}, 11:34–38, 1959.

\bibitem{HP98}
Ilker Hamzaoglu and Janak~H. Patel.
\newblock Test set compaction algorithms for combinational circuits.
\newblock In {\em Proceedings of the 1998 IEEE/ACM International Conference on
  Computer-Aided Design}, ICCAD '98, page 283–289, New York, NY, USA, 1998.
  Association for Computing Machinery.

\bibitem{K72}
Richard~M. Karp.
\newblock {\em Reducibility among Combinatorial Problems}, pages 85--103.
\newblock Springer US, Boston, MA, 1972.

\bibitem{M03}
B.~Mohar.
\newblock circular colourings of edge-weighted graphs.
\newblock {\em Journal of Graph Theory}, 43:107--116, 2003.

\bibitem{NL82}
V~Neumann-Lara.
\newblock The dichromatic number of a digraph.
\newblock {\em Journal of Combinatorial Theory, Series B}, 33(3):265--270,
  1982.

\bibitem{NSS23}
Tung Nguyen, Alex Scott, and Paul Seymour.
\newblock Some results and problems on tournament structure, 2023.

\bibitem{RWCDH03}
Nicholas Rhodes, Peter Willett, Alain Calvet, James~B. Dunbar, and Christine
  Humblet.
\newblock Clip: Similarity searching of 3d databases using clique detection.
\newblock {\em Journal of Chemical Information and Computer Sciences},
  43(2):443--448, 2003.
\newblock PMID: 12653507.

\bibitem{SS20}
Alex Scott and Paul Seymour.
\newblock A survey of $\chi$‐boundedness.
\newblock {\em Journal of Graph Theory}, 95, 08 2020.

\bibitem{Z49}
A.~A. Zykov.
\newblock On some properties of linear complexes.
\newblock {\em Mat. Sb. (N.S.)}, 24(66):163--188, 1949.

\end{thebibliography}
\end{document}